\documentclass[10pt,reqno]{amsart}
\usepackage{amsfonts}
\usepackage{amssymb,color}
\usepackage{amssymb}
\usepackage{setspace}
\usepackage[pagewise]{lineno}
\usepackage{hyperref}
\hypersetup{colorlinks=true,linkcolor=black, urlcolor=blue, citecolor=blue}

\usepackage[margin=3cm, a4paper]{geometry}

\def\normo#1{\left\|#1\right\|}

\def\aabs#1{\left|#1\right|}

\def\norm#1{\|#1\|}
\def\jb#1{\langle#1\rangle}
\def\wt#1{\widetilde{#1}}
\def\wh#1{\widehat{#1}}

\newcommand{\R}{{\mathbb R}}

\newcommand{\Z}{{\mathbb Z}}
\newcommand{\ft}{{\mathcal{F}}}

\newcommand{\les}{{\lesssim}}
\newcommand{\ges}{{\gtrsim}}

\newcommand{\supp}{{\mbox{supp}}}

\def\jb#1{\langle#1\rangle}
\def\norm#1{\|#1\|}
\def\normo#1{\left\|#1\right\|}

\def\wt#1{\widetilde{#1}}
\def\wh#1{\widehat{#1}}
\def\aabs#1{\left|#1\right|}

\newcommand{\FL}{\widehat{H}}
\newcommand{\gFL}{\widehat{M}}
\newcommand{\fL}{\widehat{L}}
\newcommand{\FX}{\widehat{X}}

\newcommand{\e}{\varepsilon}

\newcommand{\EQ}[1]{\begin{equation}\begin{split} #1 \end{split}\end{equation}}
\newcommand{\EQN}[1]{\begin{equation*}\begin{split} #1 \end{split}\end{equation*}}

\newcommand{\Del}[1]{}
\newcommand{\CAS}[1]{\begin{cases} #1 \end{cases}}

\numberwithin{equation}{section}
\newtheorem{thm}{Theorem}[section]
\newtheorem{cor}[thm]{Corollary}
\newtheorem{lem}[thm]{Lemma}
\newtheorem{prop}[thm]{Proposition}

\theoremstyle{remark}
\newtheorem{rem}{Remark}

\newtheorem{theorem}{Theorem}[section]
\newtheorem{proposition}[theorem]{Proposition}

\theoremstyle{remark}
\newtheorem{remark}[theorem]{Remark}

\theoremstyle{definition}

\begin{document}

\subjclass[2020]{35Q53, 35Q55, 35E15}
	%\subjclass[2010]{35L70, 35Q55}
\keywords{Complex-valued mKdV equation, Local well-posedness, Modulation space, Fourier-Lebesgue space}
	
\title[Complex-valued mKdV in general Fourier-Lebesgue spaces]{Complex-valued solutions of the mKdV equations in generalized Fourier-Lebesgue spaces}\thanks{The second author was supported by ARC FT230100588. The third author was supported by NSFC (No. 11971503) and Program for Innovation Research in Central University of Finance and Economics.}
	
\author[Z. Chen]{Zijun Chen}
\address{School of Mathematics, Monash University, VIC 3800, Australia }
\email{zijun.chen@monash.edu}
	
\author[Z. Guo]{Zihua Guo}
\address{School of Mathematics, Monash University, VIC 3800, Australia }
\email{zihua.guo@monash.edu}
	
\author[C. Huang]{Chunyan Huang}
\address{School of Statistics and Mathematics, Central University of Finance and Economics, Beijing
		100081, China}
\email{hcy@cufe.edu.cn}
	
\begin{abstract}
We study the \emph{complex-valued} solutions to the Cauchy problem of the modified Korteweg-de Vries equation on the real line. To study the low-regularity problems, we employ a generalized Fourier-Lebesgue space  $\gFL^{s}_{r,q}(\mathbb{R})$ that unifies the modulation spaces and the Fourier-Lebesgue spaces. We then prove sharp local well-posedness results in this space by perturbation arguments using $X^{s,b}$-type spaces. Our results improve the previous one in \cite{GV}.  
\end{abstract}
	
\maketitle
	
%\tableofcontents
	
\section{Introduction}
	
We study the Cauchy problem of the \emph{complex-valued} modified Korteweg-de Vries (mKdV) equation on the real line:
\EQ{\label{eq:mkdv}
\CAS{\partial_t u+ \partial_x^3 u=\pm u^2 \partial_x u, \quad (x,t)\in \mathbb{R}^2,\\
u(x,0)=u_0(x),}}
where $u(x,t): \R\times \R \rightarrow \mathbb{C}$. The mKdV equation is a mathematical model that describes the weakly nonlinear propagation of long waves in shallow channels. 
The mKdV equation \eqref{eq:mkdv} is invariant under the following scaling transformation:
\begin{equation}\label{scaling}
u(x,t)\rightarrow u_{\lambda}(x,t):=\lambda^{}u(\lambda x, \lambda^{3}t), \quad \lambda>0.
\end{equation}
The critical Sobolev space of \eqref{eq:mkdv} is $\dot{H}^{-1/2}$, in the sense that the homogeneous Sobolev norm is invariant under the scaling transformation \eqref{scaling}.

The mKdV equation \eqref{eq:mkdv} has been extensively studied, especially for the real-valued cases. 
From the perspective of integrability, for complex-valued data, it may be more natural to consider the complex (Hirota) mKdV equation
\EQ{\label{eq:hmkdv}
\partial_t u+ \partial_x^3 u=\pm |u|^2 \partial_x u.
}
However, for the complex-valued mKdV equation \eqref{eq:mkdv}, there are also many interesting studies (e.g. see \cite{AM, Bona} and references therein). In particular, our analytic methods in this paper work for both versions of mKdV. 

We are interested in the local well-posedness (LWP) with low-regularity data of the mKdV equation \eqref{eq:mkdv}. In general, the purpose is to obtain LWP of \eqref{eq:mkdv} for initial data in $X$ with $X$ as large as possible. Many physically interesting data are in low-regularity spaces such as the Dirac measure and white noise. 
The first natural choice of $X$ is the Sobolev space $H^{s}(\mathbb{R})$, where the norm is defined by 
\begin{equation}\label{spa_H}
\|f\|_{H^{s}(\mathbb{R})}:=\left\|\langle\xi\rangle^{s} \widehat{f}(\xi)\right\|_{L_{\xi}^{2}}
\end{equation}
for $s \in \mathbb{R}$ and $\langle\cdot\rangle:=\left(1+|\cdot|^{2}\right)^{1/2}$. For $C^1$ well-posedness (requiring the solution map to be $C^1$ smooth), the optimal range is $s\geq 1/4$, see \cite{KPV2} for LWP, and \cite{CKSTT, Guo, Ki} for GWP. The solution map of \eqref{eq:mkdv} fails to be locally uniformly continuous in $H^{s}(\mathbb{R})$ for $s<1/4$, see \cite{CCT, KPV3}.
Hence, the iteration arguments are not applicable for well-posedness of \eqref{eq:mkdv} in this regime. 
In a recent breakthrough \cite{HKV}, the optimal $C^{0}$ GWP for \eqref{eq:hmkdv} in $H^{s}(\mathbb{R})$ for $s>-1/2$ was obtained by using the complete integrability of the equation.  Also see \cite{Forlano} for a simplified proof of GWP in $H^s$ for $s\geq 0$.  However, the methods do not work for the complex-valued mKdV \eqref{eq:mkdv}, but work for \eqref{eq:hmkdv}.  Indeed, blowup solutions for \eqref{eq:mkdv} were constructed in \cite{Bona}.
	
To look for rough solutions, some other types of function spaces are used. In the Fourier-Lebesgue space $\FL^s_r$, where
\begin{equation}\label{spa_FL}
\|f\|_{\FL^s_r(\mathbb{R})}:=\left\|\langle\xi\rangle^{s} \widehat{f}(\xi)\right\|_{L_{\xi}^{r^{\prime}}},\quad 1/r+1/r^{\prime}=1
\end{equation}
for $s \in \mathbb{R}, 1 \leq r \leq \infty$,
Grünrock \cite{Gr} first showed LWP of \eqref{eq:mkdv} and \eqref{eq:hmkdv} for $s \geq \frac{1}{2}-\frac{1}{2r}$ and $4/3<r\leq 2$. 
Moreover, the range of $s$ is sharp in the sense that the solution map to the focusing complex mKdV \eqref{eq:hmkdv} is not locally uniformly continuous in $\FL^s_r(\mathbb{R})$ for $s<\frac{1}{2}-\frac{1}{2r}$ and $1<r\leq 2$; see \cite{Gr}.
The range of $r$ was later extended in \cite{GV} to $1 < r\leq 2$. Note that $\FL^0_1(\mathbb{R})$ is critical in terms of the scaling symmetry \eqref{scaling}. 
	
Another alternative space is the modulation space $M^{s}_{r,q} (\mathbb{R})$, introduced by Feichtinger \cite{Fei} via the short-time Fourier transform. 
The norm is defined by 
\begin{equation}\label{spa_mo}
\|f\|_{M^{s}_{r,q}(\mathbb{R})}:=\left\|\langle k \rangle^{s}\|\Box_k f \|_{L^{r}_{x}}\right\|_{\ell_{k}^{q}(\mathbb{Z})}
\end{equation}
for $s \in \mathbb{R}, 1 \leq r, q \leq \infty$, where $\Box_{k}$ is a Fourier projection operator adapted to the unit interval centred at $k$ (see \eqref{box}). See \cite{book} for more examples of the applications of modulation space in various types of PDE. In \cite{CG} (see also \cite{OHW,OW}) LWP of \eqref{eq:mkdv} in $M^{s}_{2,q} (\mathbb{R})$ for $s\geq 1/4$ and $2\leq q \leq \infty$ was proved by using the Fourier restriction norm method adapted to modulation spaces and an improved $L^4$ Strichartz estimate inspired by \cite{GS}. In \cite{GH} an alternative simpler proof was given and similar results for dispersion-generalized mBO equations was obtained by proving an improvement of the classical bilinear Strichartz estimate instead of using the improved $L^4$ estimate.  

Inspired by the previous works, in this paper we consider a generalized Fourier-Lebesgue type space $\gFL^{s}_{r,q}$ that unifies the Fourier-Lebesgue and Modulation spaces, and prove {optimal} LWP of \eqref{eq:mkdv} in this space. Here
\begin{equation}\label{modu-leb-spa-int}
\|f\|_{\gFL^{s}_{r,q}(\mathbb{R})}:=\left\|\langle k \rangle^{s}\|\widehat{\Box_k f} \|_{L_{\xi}^{r'}}\right\|_{\ell_{k}^{q}(\mathbb{Z})}.
\end{equation}
This space is a special case of the family of modulation spaces, see \cite{Fei2,Fei3,Fei4} and was previously used (and called Fourier-amalgam space) for other equations in \cite{Forlano2, Forlano3}.
Our main result is 
\begin{thm}\label{th1}
Let $1< r \leq 2$ and $r^{\prime} \leq q\leq \infty$. Then the complex-valued mKdV equation \eqref{eq:mkdv} is $C^3$-locally well-posed in $\gFL^{s}_{r,q}(\mathbb{R})$ for $s\geq s(r):=\frac{1}{2}-\frac{1}{2r}$. Moreover, the data-to-solution map in $\gFL^{s}_{r,q}$ is not $C^3$ continuous at the origin if $s<\frac{1}{2}-\frac{1}{2r}$.
\end{thm}

\begin{rem}
When $r=1$, $\wh M^{0}_{1,\infty}$ is critical, namely it has the same scale as $H^{-1/2}$. We are not aware of any results concerning this case. We note that for $1<r\leq 2$, $\gFL^{s(r)}_{r,\infty}(\mathbb{R})\subset H^{-\frac{1}{2}+}$. Indeed,
\EQN{
\norm{f}_{H^{-\frac{1}{2}+}}=& \norm{\jb{\xi}^{-\frac{1}{2}+}\widehat f}_{L^2}\\
\sim & \norm{\jb{k}^{-s(r)-\frac{1}{2}+} \jb{k}^{s(r)} \widehat {\Box_k f}}_{l_k^2L^2}\les \norm{f}_{\gFL^{s(r)}_{r,\infty}}.
}
For the real-valued case, the data in $\gFL^{s(r)}_{r,q}$ is contained in \cite{HKV}. It is thus natural to ask whether one can get GWP in $\gFL^{s(r)}_{r,q}$. When $r=2$ and $q<\infty$, this was proved in \cite{OW}.
\end{rem}

Theorem \ref{th1} improves the results by Gr\"unrock and Vega \cite{GV} in $\FL^s_r(\mathbb{R})$ which corresponds to the case  $q=r^\prime$ in $\gFL^{s}_{r,q}(\mathbb{R})$.
The proof of the Theorem \ref{th1} combines the ideas in \cite{GV} using the (trilinear) Fourier restriction estimates and in \cite{GH} proving refined bilinear Strichartz estimates. Our new ingredient is an improvement of the classical bilinear Strichartz estimate as stated in Lemma \ref{lem:improve-bi}. 
	
\section{Preliminary}

We use the notation $X\lesssim Y$ for $X, Y \in \mathbb{R}$ to denote that there exists a constant $C>0$ such that $X\leq CY$.  The notation $X\sim Y$ denotes that there exist positive constants $c, C$ such that $cY \leq X \leq C Y$. 
For $a\in \R$, $a\pm$ denotes $a\pm \delta$ for any sufficiently small $\delta>0$.
We use capitalized variables $\{N, M, N_1, N_2,\cdots\}$ to denote dyadic numbers, and lower-case variables $\{i,j,k,l,m,n,\cdots\}$ to denote integers unless otherwise specified.
We will use $\widehat{u}$ to denote the standard Fourier transform $\mathcal{F}_x u$, $\mathcal{F}_t u$ or $\mathcal{F}_{t, x} u$. 

Let $\varphi\in C_0^\infty(\R)$ be a real-valued, non-negative, even, and radially-decreasing function such that $\supp \varphi\subset [-5/4, 5/4]$ and $\varphi\equiv 1$ in $[-1, 1]$.  
Define $\eta(\xi):=\frac{\varphi(\xi)}{\sum_{k\in \Z}\varphi(\xi-k)}$.  
For $\lambda>0$ and $m\in \R$, we define the frequency projection operator adapted to the interval centered at $\lambda m$ and with length $\lambda$ as
\EQ{\label{box}
\widehat{(\Box^{\lambda}_{m} f)}(\xi):=\eta({\xi}/{\lambda}-m)\widehat{f}(\xi).
}
In the following, we call it box decomposition with length $\lambda$ in the frequency space.
When $\lambda=1$, we simply write $\Box^{\lambda}_{m}=\Box_{m}$. 
Let $\chi(\xi)=\varphi(\xi)-\varphi(2\xi)$.
For a dyadic number $N\in 2^{\Z}$, we define the Littlewood-Paley projectors: $\widehat{P_1f}(\xi):=\varphi(\xi)\widehat{f}(\xi)$ and for $N>1$
\EQN{
\widehat{P_Nf}(\xi):=\chi(\xi/N)\widehat{f}(\xi), \quad \widehat{P^{\pm}_Nf}(\xi):=\chi(\xi/{N})1_{\pm \xi\geq 0}\cdot \widehat{f}(\xi).
} 

We write $\fL^r=\FL^0_r$. Define $\fL_t^p\fL_x^q$ to be a Banach space with the norm $\norm{u}_{\fL_t^p\fL_x^q}=\norm{\wh u(\xi,\tau)}_{L_\tau^{p'}L_\xi^{q'}}$ and similarly for $\fL_x^q\fL_t^p$. For $s \in \mathbb{R}$ and $1 \leq r, q \leq \infty$, $\gFL^{s}_{r,q}(\mathbb{R})$ is defined to be a Banach space with the norm
\begin{equation}\label{modu-leb-spa}
\|f\|_{\gFL^{s}_{r,q}(\mathbb{R})}:=\left\|\langle k \rangle^{s}\|\Box_k f \|_{\fL_{x}^{r}}\right\|_{\ell_{k}^{q}(\mathbb{Z})}.
\end{equation}
By the definition, we have $M^{s}_{r, q}\subset \gFL^{s}_{r,q}$ if $1 \leq r \leq 2$, $\gFL^{s}_{r,r'}=\FL^s_r$, $M^{s}_{2, q}=\gFL^{s}_{2,q}$.  Let $\omega(\xi)=\xi^3$ be the dispersion relation associated with the mKdV equation \eqref{eq:mkdv} and let $W(t)=\ft^{-1}e^{it\omega(\xi)}\ft$ be the linear propagator. 
The Fourier-Lebesgue type Bourgain space $\FX_r^{s,b}$ associated with \eqref{eq:mkdv} is defined by the norm
\EQN{
\|u\|_{\FX_r^{s,b}}:=\left\|\langle \xi\rangle^s\langle \tau-\omega(\xi)\rangle^b \widehat{u}(\xi, \tau)\right\|_{L^{r^{\prime}}_{\tau, \xi}}.
}
When $s=b=0$, we write $\FX^{s,b}_{r}$ as $\fL^{r}_{t,x}$ for simplicity.
To prove our result, we use an appropriate variant of the Bourgain space $\FX_r^{s,b}$ adapted to modulation space, whose norm is defined as
\begin{equation}\label{mo-cha}
\begin{aligned}
\|u\|_{\FX^{s,b}_{r,q}}:&= \left\| \left\|\Box_k u\right\|_{\FX_r^{s,b}}\right\|_{\ell_k^{\,q}(\mathbb{Z})}\sim \left\|\langle k\rangle^s\left\|\eta(\xi-k)\langle \tau-\omega(\xi)\rangle^b \widehat{u}(\xi, \tau)\right\|_{L^{r^{\prime}}_{\tau, \xi}}\right\|_{\ell_k^{\,q}(\mathbb{Z})}.
\end{aligned}
\end{equation}
We see that $\FX^{s,b}_{r,r'}=\FX^{s,b}_r$. 
The dual space $(\FX^{s,b}_{r, q})^{\prime}$ of $\FX^{s,b}_{r, q}$ can be identified as $\FX^{-s,-b}_{r^\prime, q^\prime}$. 

We collect some basic properties of the space $\FX^{s,b}_{r, q}$. 
For $1\leq q\leq p\leq \infty$, we have 
\begin{equation}\label{Holder}
\|P_N u\|_{\FX^{s,b}_{r, p}}\leq \|P_N u\|_{\FX^{s,b}_{r, q}}\lesssim N^{1/q-1/p}\|P_N u\|_{\FX^{s,b}_{r, p}}
\end{equation}
and for $1\leq r\leq 2$
\begin{equation}\label{Holder2}
\|P_N u\|_{\widehat{L}_{x}^{r'} \widehat{L}_{t}^{\infty}}= \|\wh{P_N u}\|_{L_\xi^r L_\tau^1}\lesssim N^{1/r-1/r'}\|P_N u\|_{\FX^{0,1/r+}_{r}}.
\end{equation}
The following lemma allows us to `transfer' linear estimates in $\fL_x^r$ to estimates in $\FX^{0, b}_{r}$.

\begin{lem}[Transference Principle]\label{lem:ext}
Let $Z$ be any space-time Banach space satisfying the time modulation estimate
\EQ{
\left\|g(t)F(x,t)\right\|_Z\leq \|g\|_{L_t^{\infty}(\R)}\norm{F(x, t)}_Z }
for any $F\in Z$ and $g\in L_t^{\infty}(\R)$. Let $T:(h_1, \cdots, h_k)\rightarrow T(h_1, \cdots, h_k)$ be a spatial multilinear operator for which one has the estimate
\EQ{\left\|T(W(t)f_{1},\cdots, W(t)f_{k})\right\|_Z\les \prod_{j=1}^k \|f_{j}\|_{\fL_x^r}}
for all $f_{1},\cdots, f_{k}\in {\fL^r_x}$. Then for $b>1/r$, we have the estimate
\EQ{\|T(u_1, \cdots, u_k)\|_{Z}\les_k \prod_{j=1}^k\|u_j\|_{\FX_r^{0,b}}}
for all $u_1, \cdots, u_k \in \FX_r^{0,b}$.
\end{lem}
	
\begin{proof}
When $k=1$, this was proved in Lemma 2.1 \cite{Gr}. For $k\geq 2$, one can prove it by slightly modifying the proof in Lemma 4.1 \cite{Tao2} and Lemma 2.1 \cite{Gr}. We omit the details. 
\end{proof}

A straightforward consequence of Lemma \ref{lem:ext} is the embedding 
\EQ{\FX_{r, q}^{s, b}\subset C(\mathbb{R}; \gFL^{s}_{r,q})}
for all $b>1/r$.

\begin{lem}[Linear estimate, \cite{Gr}]\label{linear}{\rm (1)} Given $s, b \in \mathbb{R}$ and $1 \leq r,q \leq \infty$, we have
\begin{equation}
\left\|\varphi(t)W(t)f\right\|_{\FX_{r, q}^{s, b}} \lesssim\|f\|_{\gFL_{r,q}^{s}}.
\end{equation}
{\rm (2)} Given $s\in \mathbb{R}$, $1 < r<\infty$, $1 \leq q \leq \infty$ and $b{^{\prime}}+$ $1 \geq b \geq 0 \geq b^{{\prime}}>-1 / r{^{\prime}}$, we have
\begin{equation}
\left\|\varphi(t/T)\int_{0}^{t} W(t-t') F(t^{{\prime}}) d t^{{\prime}}\right\|_{\FX_{r, q}^{s, b}} \lesssim T^{1+b^{{\prime}}-b}\|F\|_{\FX_{r,q}^{s, b^{{\prime}}}}
\end{equation}
for any $0<T \leq 1$.
\end{lem}

By the standard iteration arguments (see \cite{Tao, Gr}), to prove Theorem \ref{th1} it suffices to establish the following nonlinear estimate.
\begin{prop}\label{prop:tri}
Assume that $1< r \leq 2$, $r^{\prime} \leq q\leq \infty$, and $s\geq s(r)=\frac{1}{2}-\frac{1}{2r}$. Then for $0<\varepsilon \ll 1$ we have
\begin{equation}\label{tri_1}
\left\|\varphi(t)\partial_{x}(u_1u_2u_3)\right\|_{\FX_{r, q}^{s, -1/r^{\prime}+2\varepsilon}}\lesssim \prod^3_{j=1}\|u_j\|_{\FX_{r,q}^{s, 1/r+\varepsilon}}.
\end{equation}
\end{prop}

\begin{remark}
When $q$ is smaller, the above estimate is easier to prove. Indeed, besides to gain the regularity, we need to overcome the loss in the following convolution equality
\EQ{
l^q*l^q*l^q\to l^q. 
}
By Young's inequality, we only have boundedness from $l^1*l^1*l^q\to l^q$. Therefore, there is a loss of volume $|E_1|^{1/q'} |E_2|^{1/q'}$ of the input functions in using the H\"older inequality. When $q$ is bigger, the loss is worse.  We have to exploit dedicate decoupling properties among the frequency interactions. 
\end{remark}

\section{Multilinear Strichartz estimates}

In this section, we collect various linear and multilinear Strichartz and Fourier restriction estimates that are crucial in the next section to prove Proposition \ref{prop:tri}.  The first is the linear estimate.

\begin{lem}[Corollary 3.6 and 3.7 \cite{Gr}]\label{lem:str}
Assume $\frac{2}{p}+\frac{1}{q}=\frac{1}{r}$.
Then the estimate
\begin{equation}\label{corollary3.6}
\left\|D^{1/p}W(t)f\right\|_{L_t^pL_x^q}\les \|f\|_{\fL_{x}^{r}}
\end{equation}
holds true if one of the following conditions is fulfilled:
\EQN{
&(i)\quad  0\leq 1/p \leq 1/4, 0\leq 1/q < 1/4,\\
&(ii)\quad 1/4 \leq 1/q \leq 1/q+1/p <1/2,\\
&(iii)\quad (p,q)=(\infty, 2).
}
\end{lem}

\begin{remark}
When $p=q=3r$, the conditions reduce to $r>4/3$. This is the famous Fourier restriction estimate due to Fefferman and Stein \cite{Fe-Stein}. 
Moreover, by Lemma \ref{lem:ext} we have for $4/3<r\leq 2$
\EQ{\label{stri-pq}
\left\|D^{\frac{1}{3r}}u\right\|_{L_{t,x}^{3r}}\les \left\|u\right\|_{\FX_r^{0, 1/r+}}.
}
\end{remark}

Let $m$ be a function defined on $\R^2$. 
We use $B_m(f,g)$ to denote the Coifman-Meyer bilinear multiplier
\EQ{\label{coifman}
\mathcal{F}_{x}({B_m(f,g)})(\xi)=\int m(\xi_1, \xi-\xi_1)\widehat{f}(\xi_1)\widehat{g}(\xi-\xi_1)d\xi_1.
}
We recall the bilinear estimates in the Fourier-Lebesgue spaces.
	
\begin{lem}[Bilinear estimates, Lemma 1 \cite{GV}]\label{lem:bilinear-FL}
Let
\begin{align}
m(\xi_1, \xi_2)=\left(|\xi_1+\xi_2||\xi_1-\xi_2|\right)^{1/p}.
\end{align}
Then
\EQ{\label{eq:bilinearStrichartz}
\left\|B_m(W(t)f, W(t)g)\right\|_{\fL^{q}_{x}\fL^{p}_{t}}\lesssim \|f\|_{\fL^{r_1}_{x}}\|g\|_{\fL^{r_2}_{x}}.
}
Moreover, by Lemma \ref{lem:ext}, we have for any $b_{i}>1/r_{i}$
\EQ{ \label{bi-FL}
\|B_m(u, v)\|_{\fL^{q}_{x}\fL^{p}_{t}}\lesssim \|u\|_{\FX_{r_1}^{0, b_{1}}}\|v\|_{\FX_{r_2}^{0, b_{2}}},
}
provided that $1\leq q\leq r_1, r_2 \leq p \leq \infty$ and $1/p+1/q=1/r_1+1/r_2$.
\end{lem}
	
The following corollary follows from Lemma \ref{lem:bilinear-FL}.
\begin{cor} \label{cor:bi1}
{\rm (1)} Let $N_1, N_2\geq 1$ be dyadic numbers and $N_1\gg N_2$. Then for $b>1/r$
\begin{equation}\label{cor:dy-bi}
\left\|P_{N_1}uP_{N_2}v\right\|_{{\fL^r_{t,x}}}\lesssim N_1^{-2/r}\left\|P_{N_1}u\right\|_{\FX_r^{0,b}}\left\|P_{N_2}v\right\|_{\FX_r^{0,b}}.
\end{equation}
{\rm (2)} Suppose that $k_1, k_2 \in \mathbb{Z}$ and $|k_1+k_2|, |k_1-k_2|\geq 2$. Then for $b>1/r$
\begin{equation} \label{cor:mo-bi}
\left\|\Box_{k_1}u\Box_{k_2}v\right\|_{{\fL^r_{t,x}}}\lesssim B(k_1,k_2)\left\|\Box_{k_1}u\right\|_{\FX_r^{0,b}}\left\|\Box_{k_2}v\right\|_{\FX_r^{0,b}},
\end{equation}
where $B(k_1,k_2)=(|k_1+k_2||k_1-k_2|)^{-1/r}$.
\end{cor}

The bilinear Strichartz estimate \eqref{eq:bilinearStrichartz} is an equality when $q=p=r_1=r_2=r$.  However, after a temporal localization, we can have an improvement by some decoupling properties. This was first observed in \cite{GH} in the case $r=2$. We extend this to general case $r>1$.  This is a key ingredient in the proof of Proposition \ref{prop:tri}. 
\begin{lem}[Improved bilinear estimate]\label{lem:improve-bi}
Assume $1\leq N_3\ll N_1\sim N_2$ and  \EQ{m_\pm(\xi_1,\xi_2)=\chi_{N_3}(\xi_1\pm \xi_2)\chi_{N_1}(\xi_1)\chi_{N_2}(\xi_2).}
Then for $2r'\leq q\leq \infty$ and $b>1/r$ we have
\EQ{\label{improve-bi}
&\left\|B_{m_-}(u, v)\right\|_{{\fL^r_{t,x}}}+\left\|{\varphi}(t)B_{m_+}(u, v)\right\|_{{\fL^r_{t,x}}}\\
\les& N_1^{-1/r}N_1^{1/r^{\prime}-2/q}N_3^{1/r^{\prime}-1/r}\left\|u\right\|_{\FX^{0,b}_{r,q}}\left\|v\right\|_{\FX^{0,b}_{r,q}}.
}
\end{lem}
\begin{proof}
First, we consider the symbol $m_-$. We decompose $m_-$ as
\EQN{
m_-(\xi_1,\xi_2)=&\sum_j\sum_{j_1,j_2\in \Z} \chi_{N_3}(\xi_1-\xi_2)\chi_{N_1}(\xi_1)\chi_{N_2}(\xi_2)\\
& \cdot\chi (\frac{j_1N_3+\xi_1}{N_3})\chi (\frac{j_2N_3+\xi_2}{N_3})\chi (\frac{jN_3+\xi_1+\xi_2}{N_3})\\
=&\sum_j m_-^j(\xi_1,\xi_2).
}
By the support properties of the functions, we know that in the above summations $|j_1-j_2|\leq 3$ and $|j-j_1-j_2|\leq 5$, $j$ takes $O(N_1/N_3)$ values. Hence for the fixed $j$, we have $|j_i-j/2|\leq 10$, $i=1,2$. 
For simplicity, we assume $j_1=j_2=j/2$.
Then for $q\geq 2r'$ we have
\EQN{
&\left\|B_{m_-}(W(t)f, W(t)g)\right\|_{{\fL^r_{t,x}}}\\
\les &\left\|\mathcal{F}(B_{m^j_-}(W(t)f, W(t)g))\right\|_{l_j^{r^{\prime}} L^{r^{\prime}}_{\tau,\xi}}\\
\les & N_1^{-1/r}N_3^{-1/r}\normo{\|\chi (\frac{jN_3/2+\xi}{N_3})\widehat{f}\|_{L_{\xi}^{r^{\prime}}}\|\chi (\frac{jN_3/2+\xi}{N_3})\widehat{ g}\|_{L_{\xi}^{r^{\prime}}}}_{l_j^{r^{\prime}}}\\
\les & N_1^{-1/r}N_3^{-1/r}(N_1/N_3)^{1/{r^{\prime}}-2/q}\normo{\|\chi (\frac{jN_3/2+\xi}{N_3})\widehat{f}\|_{L_{\xi}^{r^{\prime}}}\|\chi (\frac{jN_3/2+\xi}{N_3})\widehat{ g}\|_{L_{\xi}^{r^{\prime}}}}_{l_j^{q/2}}\\
\les & N_1^{-1/r}N_3^{-1/r}(N_1/N_3)^{1/{r^{\prime}}-2/q}N_3^{2(1/{r^{\prime}}-1/q)}\|\widehat{\Box_k f}\|_{l_k^q L_{\xi}^{r^{\prime}}}\|\widehat{\Box_k g}\|_{l_k^q L_{\xi}^{r^{\prime}}}\\
\les & N_1^{-1/r}N_1^{1/{r^{\prime}}-2/q}N_3^{1/{r^{\prime}}-1/r}\|\widehat{\Box_k f}\|_{l_k^q L_{\xi}^{r^{\prime}}}\|\widehat{\Box_k g}\|_{l_k^q L_{\xi}^{r^{\prime}}},
}
where we used Lemma \ref{lem:bilinear-FL} in the second step. 
Thus we have \eqref{improve-bi} as desired by Lemma \ref{lem:ext}.
		
Now we consider the symbol $m_+$. We decompose the high frequency into the intervals with length $N_3$
\EQN{
&{\varphi}(\tau)\mathcal{F}(B_{m_+}(W(t) f, W(t)g))\\
=&{\varphi}(\tau) \mathcal{F}(P_{N_3}(W(t)P_{N_1}f\cdot W(t)P_{N_2}g))\\
=&\sum_{|j_1+j_2|\leq 3}{\varphi}(\tau)\mathcal{F}(P_{N_3}(W(t)P_{N_1}\Box^{N_3}_{j_1}f\cdot W(t)P_{N_2}\Box^{N_3}_{j_2}g)).
}
Without loss of generality, we may assume $j_1=-j_2$. The key observation is the following almost orthogonality property
\EQ{\label{eq:improvepf1}
&\left\|\sum_j {\varphi}(\tau)\mathcal{F}(P_{N_3}(W(t)P_{N_1}\Box^{N_3}_{j}f\cdot W(t)P_{N_2}\Box^{N_3}_{-j}g))\right\|_{L^{r^{\prime}}_{\tau, \xi}}\\
\les& \left\|{\varphi}(\tau)\mathcal{F}(P_{N_3}(W(t)P_{N_1}\Box^{N_3}_{j}f\cdot W(t)P_{N_2}\Box^{N_3}_{-j}g))\right\|_{l_j^{r^{\prime}}L^{r^{\prime}}_{\tau, \xi}}.
}
Assuming \eqref{eq:improvepf1}, then we have the desired estimates as the case $m_-$.
To show \eqref{eq:improvepf1}, one only needs to note that ${\varphi}(\tau)\mathcal{F}[P_{N_3}(W(t)P_{N_1}\Box^{N_3}_{j}f\cdot W(t)P_{N_2}\Box^{N_3}_{-j}g)]$ is supported in the set $E_j$ where
\EQN{
E_j=\{(\tau,\xi): |\xi|\sim N_3, |\tau-3|j N_3|^{2}\xi|\les N_1 N_3^2, |jN_3|\sim N_1\}.
}
It's easy to verify that $\{E_j\}$ is finitely overlapping. 
\end{proof}

\begin{remark}
Lemma \ref{lem:bilinear-FL} implies that for $r' \leq q \leq \infty$
\EQ{
&\|B_{m_-}(u, v)\|_{{\fL^r_{t,x}}}+\|{\varphi}(t)B_{m_+}(u, v)\|_{{\fL^r_{t,x}}}\\
\les&N_1^{-1/r}N_3^{-1/r}N_1^{2/{r'}-2/q}\|u\|_{\FX^{0,b}_{r,q}}\|v\|_{\FX^{0,b}_{r,q}}\\
=& N_1^{-1/r}N_1^{1/{r'}-2/q}N_3^{1/r'-1/r} (N_1^{1/{r'}}N_3^{-1/{r'}})\|u\|_{\FX^{0,b}_{r,q}}\|v\|_{\FX^{0,b}_{r,q}}.
}
However, by Lemma \ref{lem:improve-bi} we gain a factor $N_1^{1/{r'}}N_3^{-1/{r'}}$ for $q\geq 2r'$. This will be crucial for us to handle the case $N_3\ll N_1\sim N_2$.
    
\end{remark}
 
Lemma \ref{lem:improve-bi} with $r=2$ can imply the improved $L^4$-estimates (see \cite{GH})
\EQ{\label{L4-1}
\norm{\varphi(t)P_Nu}_{L^4_{t,x}}\les N^{-1/8+} \norm{u}_{\FX^{0, 1/2+}_{2,4}}.
}
Taking $r=4/3+$ in \eqref{stri-pq} and using H\"older inequality, we have
\begin{align}\label{L4-estimate3}
\|\varphi(t) P_N u\|_{L^4_{t,x}}\les N^{-1/4+}\|u\|_{\FX_{4/3, 4}^{0, 3/4+}}.
\end{align}
By the interpolation of \eqref{L4-estimate3} with \eqref{L4-1} we obtain for $4/3< r \leq 2$
\EQ{\label{L4-2}
\norm{\varphi(t)P_Nu}_{L^4_{t,x}}\les N^{-(\frac{1}{2r}-\frac{1}{8})+}  \norm{u}_{\FX^{0,1/r+}_{r, 4}}.
}
The above estimate is useful to handle the resonant case. For $1<r\leq 4/3$, we will use the trilinear estimates in \cite{GV}. 

Let $\sigma$ be a general function on $\mathbb{R}^{3}$. Define the trilinear multiplier by
\EQN{
\mathcal{F}_{x}({T_\sigma (f_1, f_2, f_3)})(\xi)=\iint \sigma (\xi_1, \xi_2, \xi-\xi_1-\xi_2)\widehat{f_1}(\xi_1)\widehat{f_2}(\xi_2)\widehat{f_3}(\xi-\xi_1-\xi_2)d\xi_1d\xi_2.
}

\begin{lem}[Trilinear estimates]\label{lem:trilinear-FL}

(a) (Corollary 5 \cite{GV}): For $1\leq r<\infty$, $b>1/r$ and $\varepsilon>0$, we have
\EQ{
\normo{T_\sigma (u_1, u_2, u_3)}_{\fL_{t,x}^{r}}\lesssim \|u_1\|_{\FX^{\varepsilon, b}_{r}}\normo{u_2}_{\FX^{-\frac{1}{2r}, b}_{r}}\|u_3\|_{\FX^{-\frac{1}{2r}, b}_{r}}
}
where
\begin{align}
\sigma (\xi_1, \xi_2, \xi-\xi_1-\xi_2)=\chi\{|\xi-\xi_1|\leq |\xi_1+2\xi_2-\xi|\}.
\end{align}

(b) (Corollary 4 \cite{GV}): For $1<r<2$, $b>1/r$, we have
\EQ{\label{eq:trilinearStrichartz}
\normo{T_\sigma (u_1, u_2, u_3)}_{\fL_{t,x}^{r}}\lesssim \|u_1\|_{\FX^{-\frac{1}{3r}, b}_{r}}\normo{u_2}_{\FX^{-\frac{1}{3r}, b}_{r}}\|u_3\|_{\FX^{-\frac{1}{3r}, b}_{r}}
}
where
\begin{align}
\sigma (\xi_1, \xi_2, \xi-\xi_1-\xi_2)=\chi\{|\xi-\xi_1|\geq |\xi_1+2\xi_2-\xi|\}.
\end{align}
\end{lem}

\section{Proof of Proposition \ref{prop:tri}}

In this section we prove Proposition \ref{prop:tri}. 
Fix $r\in (1,2]$ and $0<\e< \frac{(r-1)^2}{2^{10} r^2}$. By duality, to prove \eqref{tri_1} it is equivalent to show
\begin{align}\label{tri7}
\left|\int \wh{\varphi}^4(t) u_1u_2u_3v\, dxdt\right|&\lesssim \prod^3_{j=1}\|u_j\|_{\FX_{r, q}^{s, 1/r+\varepsilon}}\|v\|_{\FX_{r^{\prime}, q^{\prime}}^{-1-s,  1/r^{\prime}-2\varepsilon}}.
\end{align}
To simplify notations, we write $u_{N_j}={\wh\varphi}(t)P_{N_j}u_j$ for $j=1,2,3$, and $v_{N}={\wh\varphi}(t)P_Nv$. Applying the Littlewood-Paley dyadic decomposition to each component of \eqref{tri7}, we first prove
\begin{align}\label{tri7-2}
\left|\int u_{N_1}u_{N_2}u_{N_3}v_N\, dxdt\right|&\leq C(N_1,N_2,N_3, N) \prod^3_{j=1}\|u_{N_j}\|_{\FX_{r, q}^{s, 1/r+\varepsilon}}\|v_N\|_{\FX_{r^{\prime}, q^{\prime}}^{-1-s, 1/r^{\prime}-2\varepsilon}},
\end{align}
where $C(N_1, N_2, N_3, N)$ is a suitable bound that allows us to sum over all dyadic numbers.

A straightforward computation yields
\EQN{
\int u_{N_1}u_{N_2}u_{N_3}v_N\, dxdt=\int_{\Gamma}\prod_{j=1}^3\wh{u}_{N_j}(\xi_j,\tau_j)\wh{v}_{N}(\xi,\tau)d\mu,
}
where $\Gamma=\{\xi_1+\xi_2+\xi_3+\xi=0, \tau_1+\tau_2+\tau_3+\tau=0\}$ and $d\mu$ denotes the induced Lebesgue measure. So
\EQ{\label{eq:4linear}
\aabs{\int u_{N_1}u_{N_2}u_{N_3}v_N\, dxdt}\les & \int_{\Gamma}\prod_{j=1}^3|\wh{u}_{N_j}(\xi_j,\tau_j)|\cdot |\wh{v}_{N}(\xi,\tau)|\, d\mu.
}
For a function $u$ defined on $\R^2$, we define a transform as
\EQN{
\wt{u}(x,t):=\ft^{-1}|\wh{u}(\xi,\tau)|.
}
Note that $\norm{u}_{\FX^{s,b}_{r,q}}=\norm{\wt{u}}_{\FX^{s,b}_{r,q}}$ and
\EQ{
\int_{\Gamma}\prod_{j=1}^3|\wh{u}_{N_j}(\xi_j,\tau_j)|\cdot |\wh{v}_{N}(\xi,\tau)|\, d\mu \sim \aabs{\int \wt{u}_{N_1} \wt u_{N_2} \wt u_{N_3} \wt v_N\, dxdt}.
}
This allows us to put absolute values on both physical and frequency sides. 
To prove \eqref{tri7-2}, it suffices to prove 
\begin{align}\label{tri7-3}
\left|\int \wt{u}_{N_1} \wt u_{N_2} \wt u_{N_3} \wt v_N\, dxdt\right|&\leq C(N_1,N_2,N_3, N) \prod^3_{j=1}\|u_{N_j}\|_{\FX_{r, q}^{s, 1/r+\varepsilon}}\|v_N\|_{\FX_{r^{\prime}, q^{\prime}}^{-1-s, 1/r^{\prime}-2\varepsilon}}.
\end{align}
By symmetry, we may assume $N_1\geq N_2\geq N_3$.  We may also assume that $N_1\sim \max(N_2, N)$ by the frequency support properties of the functions. Then we need to make sure the bound is summable, namely
\EQ{
\sum_{N_1,N_2,N_3,N: N_1\sim \max(N_2,N)}C(N_1,N_2,N_3, N) &\prod^3_{j=1}\|u_{N_j}\|_{\FX_{r, q}^{s, 1/r+\varepsilon}}\|v_N\|_{\FX_{r^{\prime}, q^{\prime}}^{-1-s, 1/r^{\prime}-2\varepsilon}}\\
\les & \prod^3_{j=1}\|u_j\|_{\FX_{r, q}^{s, 1/r+\varepsilon}}\|v\|_{\FX_{r^{\prime}, q^{\prime}}^{-1-s,  1/r^{\prime}-2\varepsilon}}.
}

Let $\sigma_j$ and $\sigma$ denote the modulations given by
$$\sigma_j=\tau_j-\xi_j^3, \qquad  \sigma=\tau-\xi^3.$$
Under the restrictions $\xi_1+\xi_2+\xi_3+\xi=0$ and $\tau_1+\tau_2+\tau_3+\tau=0$, we have
\begin{equation*}
-\tau+\xi^3=\tau_1-\xi_1^3+\tau_2-\xi_2^3+\tau_3-\xi_3^3 -\Phi(\xi_1, \xi_2, \xi_3)
\end{equation*}
where $\Phi$ is the resonance function
\EQ{
\Phi(\xi_1, \xi_2, \xi_3)=(\xi_1+\xi_2+\xi_3)^3-(\xi_1^3+\xi_2^3+\xi_3^3)=3(\xi_1+\xi_2)(\xi_2+\xi_3)(\xi_1+\xi_3).
}
Let $$\sigma_{max}=\max\{|\sigma_1|,|\sigma_2|,|\sigma_3|,|\sigma|\}.$$
Thus we may assume 
\begin{equation}\label{sigma_max}
\sigma_{max}\gtrsim |\Phi|
\end{equation}
in the domain $\Gamma$ of the integral \eqref{eq:4linear}.

We will prove \eqref{tri7-3} case by case. 

{\bf Case 1:} $N_1\lesssim 1$ or $N_1\gg 1$ with $\sigma_{max} \gg N_1^{20 r'}$. These are trivial cases.
	
{$\bullet$} $N_1\lesssim 1$, then $N\les 1$. 
We have
\begin{align*}
&\left|\int \wt u_{N_1} \wt u_{N_2} \wt u_{N_3} \wt v_N\, dxdt\right|\\
\lesssim& \|\wt u_{N_1}\| _{\widehat{L}^{r}_{t,x}}\|\wt u_{N_2}\| _{{\widehat{L}^{\infty}_{t,x}}}\|\wt u_{N_3}\| _{{\widehat{L}^{\infty}_{t,x}}}\|\wt v_{N}\|_{\widehat{L}^{r^{\prime}}_{t,x}}\\
\lesssim&\|\wt u_{N_1}\|_{\FX_{r,q}^{s,1/r+}}\|\wt u_{N_2}\|_{\FX_{r,q}^{s,1/r+}}\|\wt u_{N_3}\|_{\FX_{r,q}^{s,1/r+}}\|\wt v_{N} \|_{\FX_{r^{\prime}, q^{\prime}}^{-1-s, 1/r^{\prime}-2\e}},
\end{align*}
which is summable.
	
{$\bullet$} $N_1\gg 1$ with $\sigma_{max} \gg N_1^{20 r'}$.
We may assume $|\sigma|=\sigma_{max}$, and the other cases follow by a similar argument. 
By H\"older's inequality, we have
\begin{align*}
&\left|\int \wt u_{N_1} \wt u_{N_2} \wt u_{N_3} \wt v_Ndxdt\right|\\
\lesssim &\|\wt u_{N_1}\| _{\widehat{L}^{r}_{t,x}}\|\wt u_{N_2}\| _{\widehat{L}^{\infty}_{t,x}}\|\wt u_{N_3}\| _{\widehat{L}^{\infty}_{t,x}}\|\wt v_{N}\|_{\widehat{L}^{r^{\prime}}_{t,x}}\\
\lesssim& N_1^{-2}\|\wt u_{N_1}\|_{\FX_{r,q}^{s,1/r+}}\|\wt u_{N_2}\|_{\FX_{r,q}^{s,1/r+}}\|\wt u_{N_3}\|_{\FX_{r,q}^{s,1/r+}}\|\wt v_{N} \|_{\FX_{r^{\prime}, q^{\prime}}^{-1-s, 1/r^{\prime}-2 \e}},
\end{align*}
which is summable.  
	
Therefore, in what follows we only need to consider
\begin{equation}
N_1 \gg 1 \quad \text{and} \quad \sigma_{max} \les N_1^{20 r'}.
\end{equation}

{\bf Case 2:} $N_{1}\gg N_{2}\geq N_{3}$. 
	
In this case, we have $N_1\sim N$ and $|\Phi|\sim N_1^{2}|\xi_2+\xi_3|=N_1^{2}|\xi_1+\xi|$.
By \eqref{sigma_max} we may assume $\sigma_{max}\ges N_1^{2} |\xi_1+\xi|$ in $\Gamma$ and obtain
\EQ{\label{tri-case2}
\aabs{\int \wt{u}_{N_1} \wt u_{N_2} \wt u_{N_3} \wt v_N\, dxdt} \leq \sum_{K}\aabs{\int P_K(\wt{u}_{N_1}\wt v_N) \wt u_{N_2} \wt u_{N_3}dxdt}. 
}
	
{$\bullet$} $K\les 1$ in \eqref{tri-case2}. Since $|\xi_1+\xi|=|\xi_2+\xi_3|\les 1$, we have
\begin{equation}\label{case1}
\begin{aligned}
\eqref{tri-case2}\les \sum_{|k_1+k|\les 1}\int \left|(\Box_{k_1}\wt{u}_{N_1})\wt{u}_{N_2}\wt{u}_{N_3}(\Box_k\wt{v}_N)\right|dxdt.
\end{aligned}
\end{equation}
Applying H\"older's inequality, bilinear estimate \eqref{cor:dy-bi} and \eqref{Holder}, we see that
\EQ{\label{case2a-1}
\eqref{case1}\les & \sum_{|k_1+k|\les 1}\|(\Box_{k_1}\wt{u}_{N_1})\wt{u}_{N_2}\|_{\widehat{L}^r_{t,x}}\|\wt{u}_{N_3}(\Box_k\wt{v}_{N})\|_{\widehat{L}^{r^{\prime}}_{t,x}}\\
\lesssim&  \sum_{k} N_1^{-2/r}\|\Box_{-k}\wt{u}_{N_1}\|_{\FX_r^{0,1/r+}}\|\wt{u}_{N_2}\|_{\FX_r^{0,1/r+}}N_1^{-2/r^{\prime}}\|\wt{u}_{N_3}\|_{\FX_{r^{\prime}}^{0,1/r^{\prime}+}}\|\Box_k\wt{v}_{N}\|_{\FX_{r^{\prime}}^{0,1/r^{\prime}+}}\\
\lesssim& \sum_{k}N_1^{-2}N_3^{1/r-1/r^{\prime}}\|\Box_{-k}\wt{u}_{N_1}\|_{\FX_r^{0,1/r+}}\|\wt{u}_{N_2}\|_{\FX_r^{0,1/r+}}\|\wt{u}_{N_3}\|_{\FX_{r}^{0,1/r+}}\|\Box_k\wt{v}_{N}\|_{\FX_{r^{\prime}}^{0,1/r^{\prime}+}}\\
\lesssim& N_1^{-1+\frac{1}{8r'}+}N_2^{-s+1/r^{\prime}-1/q}N_3^{-s+1/r-1/q}\\	&\cdot\|u_{N_1}\|_{\FX_{r,q}^{s,1/r+}}\|u_{N_2}\|_{\FX_{r,q}^{s,1/r+}}\|u_{N_3}\|_{\FX_{r, q}^{s,1/r+}}\|v_{N}\|_{\FX_{r^{\prime}, q^{\prime}}^{-1-s,1/r^{\prime}-2\e}},
}
where in the last step of \eqref{case2a-1} we used the fact that
\begin{equation}
\|v_{N}\|_{\FX_{r^{\prime}, q^{\prime}}^{-1-s,1/r^{\prime}+}}\les N_{1}^{\frac{1}{8r'}}\|v_{N}\|_{\FX_{r^{\prime}, q^{\prime}}^{-1-s,1/r^{\prime}-2\e}}
\end{equation}
since $\sigma_{max} \les N_1^{20 r'}$.  If $-s+1/r-1/q\leq 0$, the above bound is summable. If $-s+1/r-1/q > 0$, then
\begin{equation*}
\eqref{case2a-1}\les N_1^{-1+\frac{1}{8r'}+} N_2^{-2s+1-2/q}\|u_{N_1}\|_{\FX_{r,q}^{s,1/r+}}\|u_{N_2}\|_{\FX_{r,q}^{s,1/r+}}\|u_{N_3}\|_{\FX_{r, q}^{s,1/r+}}\|v_{N}\|_{\FX_{r^{\prime}, q^{\prime}}^{-1-s,1/r^{\prime}-2\e}}
\end{equation*}
 which is summable if {$s>\frac{1}{16 r'}-\frac{1}{q}$}.

{$\bullet$} $K\ges 1$ in \eqref{tri-case2}. 
Since $|\xi_1+\xi|=|\xi_2+\xi_3|\sim K \ges1$, we take the box decomposition with length $K$ on the terms $\wt{u}_{N_1}$ and $\wt{v}_{N}$. Then we have
\begin{equation}\label{case2b}
\begin{aligned}		
\eqref{tri-case2}\les&{\sum_{K}}\sum_{|k_1+k|\les 1} \int \left|(\Box^K_{k_1} \wt{u}_{N_1})\wt{u}_{N_2}\wt{u}_{N_3}(\sum_m\Box_m\Box_k^K\wt{v}_N)\right|dxdt\\
\les&{\sum_{K}}\sum_{k}\int \left|(\Box^K_{-k} \wt{u}_{N_1})\wt{u}_{N_2}\wt{u}_{N_3}(\sum_m\Box_m\Box_k^K\wt{v}_N)\right|dxdt.
\end{aligned}
\end{equation}
	
If {{$|\sigma_3|=\sigma_{max}$,}} by bilinear estimate \eqref{cor:dy-bi}, H\"older's inequality \eqref{Holder}, and the largest modulation $\sigma_{max} \ges N_1^{2} K$, we see that
\begin{equation}
\begin{aligned}
\eqref{case2b}\les&\sum_{K}\sum_{k}\sum_m\|(\Box^K_{-k}\wt{u}_{N_1})\wt{u}_{N_2}\|_{\widehat{L}^{r}_{t,x}}\|\wt{u}_{N_3}\|_{\widehat{L}^{r^{\prime}}_{t,x}}\|\Box_m\Box_k^K\wt{v}_{N}\|_{\widehat{L}^{\infty}_{t,x}}\\
\les&\sum_{K}\sum_{k}\sum_m N_1^{-2/r}N_3^{1/r-1/r^{\prime}}(N_{1}^{2}K)^{-1/r-}\\
&\cdot\|\Box^{K}_{-k}\wt{u}_{N_1}\|_{\FX_{r}^{0,1/r+}}\|\wt{u}_{N_2}\|_{\FX_{r}^{0,1/r+}}\|\wt{u}_{N_3}\|_{\FX_{r}^{0, 1/r+}}\|\Box_m\Box^{K}_{k}\wt{v}_{N}\|_{{\FX^{0, 1/r^{\prime}+}_{r^{\prime}}}}\\
\les&\sum_{K}\sum_{k} N_1^{-2/r}N_3^{1/r-1/r^{\prime}}(N_{1}^{2}K)^{-1/r-}K^{1/q}\\ &\cdot\|\Box^{K}_{-k}\wt{u}_{N_1}\|_{\FX_r^{0,1/r+}}\|\wt{u}_{N_2}\|_{\FX_r^{0,1/r+}}\|\wt{u}_{N_3}\|_{\FX_{r}^{0, 1/r+}}\|\Box^{K}_{k}\wt{v}_{N}\|_{\FX^{0,1/r^{\prime}+}_{r^{\prime},q^{\prime}}} \\
\les&\sum_{K}\sum_k K^{1/r^{\prime}-1/q}N_1^{-2/r}N_2^{1/r^{\prime}-1/q}N_3^{1/r-1/q}(N_{1}^{2}K)^{-1/r-}K^{1/q}\\ &\cdot\|\Box^{K}_{-k}u_{N_1}\|_{\FX_{r,q}^{0,1/r+}}\|u_{N_2}\|_{\FX_{r,q}^{0,1/r+}}\|u_{N_3}\|_{\FX_{r,q}^{0,1/r+}}\|\Box^{K}_{k}v_{N}\|_{\FX_{r^{\prime},q^{\prime}}^{0,1/r^{\prime}+}}\\
\les&  N_1^{1-4/r +\frac{1}{8r'}+}N_2^{-s+1/r^{\prime}-1/q}N_3^{-s+1/r-1/q}\\
&\cdot \|u_{N_1}\|_{\FX_{r,q}^{s,1/r+}}\|u_{N_2}\|_{\FX_{r,q}^{s,1/r+}}\|u_{N_3}\|_{\FX_{r,q}^{s,1/r+}}\|v_{N}\|_{\FX_{r^{\prime},q^{\prime}}^{-1-s,1/r^{\prime}-2\e}}.
\end{aligned}
\end{equation}
Therefore, the bound is summable if {$s> 1-\frac{2}{r}+\frac{1}{16r'}-\frac{1}{q}$.}

If $|\sigma_2|=\sigma_{max}$, this subcase follows the same line as in subcase $|\sigma_3|=\sigma_{max}$ and leads again to the condition {$s> 1-\frac{2}{r}+\frac{1}{16r'}-\frac{1}{q}$}.
	
If {{$|\sigma_1|=\sigma_{max}$,}} then we have
\begin{equation}\label{sigma_1}
\begin{aligned}
\eqref{case2b}
\les&\sum_{K}\sum_{k} \|\Box^{K}_{-k}\wt{u}_{N_1}\|_{\widehat{L}^{r^{\prime}}_{t,x}}\|\wt{u}_{N_2}(\sum_m\Box_m\Box^{K}_{k}\wt{v}_{N})\|_{\widehat{L}^{r}_{t,x}}\|\wt{u}_{N_3}\|_{\widehat{L}^{\infty}_{t,x}}\\
\les&\sum_{K}\sum_{k}\sum_m K^{1/r-1/r^{\prime}} (N_1^{2}K)^{-1/r-}N_1^{-2/{r}}N_3^{1/r}\\
&\cdot\|\Box^{K}_{-k}\wt{u}_{N_1}\|_{\FX_{r}^{0,1/r+}}\|\wt{u}_{N_2}\|_{\FX_{r}^{0,1/r+}}\|\wt{u}_{N_3}\|_{\FX_r^{0,1/r+}}\|\Box_{m}\Box^{K}_{k}\wt{v}_{N}\|_{\FX^{0,1/r+}_{r}}\\
\les&  N_1^{1-\frac{4}{r}+\frac{1}{8r'}+}N_2^{-s+1/r^{\prime}-1/q}N_3^{-s+1-1/q}\\
&\cdot \|u_{N_1}\|_{\FX_{r,q}^{s,1/r+}}\|u_{N_2}\|_{\FX_{r,q}^{s,1/r+}}\|u_{N_3}\|_{\FX_{r,q}^{s,1/r+}}\|v_{N}\|_{\FX_{r^{\prime},q^{\prime}}^{-1-s,1/r^{\prime}-2\e}}.
\end{aligned}
\end{equation}
This bound is summable if {$s>\frac{3}{2}-\frac{5}{2r}-\frac{1}{q}+\frac{1}{16r'}$}.

If {\underline{$|\sigma |=\sigma_{max}$,}} similarly we have                                    
\begin{equation}
\begin{aligned}
\eqref{case2b}
\les&\sum_{K}\sum_{k}\sum_m \|(\Box^{K}_{-k}\wt{u}_{N_1})\wt{u}_{N_2}\|_{\widehat{L}^r_{t,x}}\|\wt{u}_{N_3}\|_{\widehat{L}_{x}^{r'}\widehat{L}_{t}^{\infty}}\|\Box_m\Box^{K}_{k}\wt{v}_{N}\|_{\widehat{L}_{x}^{\infty}\widehat{L}_{t}^{r^{\prime}}}\\
\les&\sum_{K}\sum_{k}\sum_m N_1^{-2/r}N_3^{1/r-1/r^{\prime}}(N_{1}^{2}K)^{-1/r^{\prime}-}\\
&\cdot\|\Box^{K}_{-k}\wt{u}_{N_1}\|_{\FX_r^{0,1/r+}}\|\wt{u}_{N_2}\|_{\FX_r^{0,1/r+}} \|\wt{u}_{N_3}\|_{\FX_r^{0,1/r+}}\|\Box_m \Box^{K}_{k}\wt{v}_{N}\|_{\FX^{0, 1/ r^{\prime}+}_{r^{\prime}}}\\
\les & \sum_{K}\sum_k K^{1/r^\prime-1/q }N_1^{-2/r}N_2^{1/r^{\prime}-1/q}N_3^{1/r-1/q}(N_{1}^{2}K)^{-1/r^{\prime}-}K^{1/q}\\
&\cdot \|\Box^{K}_{-k}u_{N_1}\|_{\FX_{r,q}^{0,1/r+}}\|u_{N_2}\|_{\FX_{r,q}^{0,1/r+}}\|u_{N_3}\|_{\FX_{r,q}^{0,1/r+}}\|\Box^{K}_{k}v_{N}\|_{\FX^{0,1/r^{\prime}+}_{r^{\prime},q^{\prime}}}\\
\les & N_1^{-1+\frac{1}{8r'}+}N_2^{-s+1/r^{\prime}-1/q}N_3^{-s+1/r-1/q}\\
&\cdot\|u_{N_1}\|_{\FX_{r,q}^{s,1/r+}}\|u_{N_2}\|_{\FX_{r,q}^{s,1/r+}}\|u_{N_3}\|_{\FX_{r,q}^{s,1/r+}}\|v_{N}\|_{\FX^{-1-s,1/r^{\prime}-2\e}_{r^{\prime},q^{\prime}}}.
\end{aligned}
\end{equation}
This bound is summable if {$s>\frac{1}{16r'}-\frac{1}{q}$.}
 
{\bf Case 3:} $N_{1}\sim N_{2} \gg N_{3}$. 
	
We may assume $\sigma_{max}\ges N_1^{2}|\xi_1+\xi_2| \sim N_1^{2}K$. We have
\begin{equation}\label{case3}
\left|\int \wt u_{N_1}\wt u_{N_2}\wt u_{N_3}\wt v_Ndxdt\right|\leq \sum_K\left|\int P_{K}(\wt u_{N_1}\wt u_{N_2})\wt u_{N_3}\wt v_Ndxdt\right|.
\end{equation}

{$\bullet$} Assume $K\les 1$. 
We have $|\xi_1+\xi_2|=|\xi_3+\xi|\les 1$ and
\begin{equation*}
\begin{aligned}
\eqref{case3}
\les& \sum_{|k_3+k|\les 1}\left|\int  \wt u_{N_1}\wt u_{N_2}(\Box_{k_3} \wt u_{N_3})(\Box_k \wt v_N)dxdt\right|\\
\les& \sum_k \|\wt u_{N_1}(\Box_{-k} \wt u_{N_3})\|_{\widehat{L}^{r}_{t, x}}\|\wt u_{N_2}(\Box_k \wt v_N)\|_{\widehat{L}^{r^{\prime}}_{t, x}}\\
\les &\sum_{k}N_1^{-2/r} \|u_{N_1}\|_{\FX_r^{0,1/r+}}\|\Box_{-k}u_{N_3}\|_{\FX_r^{0,1/r+}}N_1^{-2/r^{\prime}}	\|u_{N_2}\|_{\FX_{r^{\prime}}^{0,1/r^{\prime}+}}\|\Box_kv_{N}\|_{\FX_{r^{\prime}}^{0,1/r^{\prime}+}}\\
\les& N_1^{-2s+\frac{1}{8r'}-\frac{2}{q}+}\|u_{N_1}\|_{\FX_{r,q}^{s,1/r+}} \|u_{N_2}\|_{\FX_{r,q}^{s,1/r+}}\|u_{N_3}\|_{\FX_{r,q}^{s,1/r+}}\|v_{N}\|_{\FX_{r^{\prime}, q^{\prime}}^{-1-s,1/r^{\prime}-2\e}}.
\end{aligned}
\end{equation*}
This bound is summable if {$s>\frac{1}{16r'}-\frac{1}{q}$}.
	
{$\bullet$} Assume $K\ges 1$. We assume $|\sigma_3|=\sigma_{max}$ as other cases are similar and easier. If $N_3\sim N$, then we need to localize $|\xi_3+\xi|\sim K$ if $\xi_3$ and $\xi$ have opposite signs. Therefore, whether $N_3\gg N$, $N\gg N_3$ or $N_3\sim N$, {we use the improved bilinear estimate \eqref{improve-bi} with $q=2r'$ to obtain}
\begin{equation}
\begin{aligned}
\eqref{case3}
\les&{\sum_{K}}\sum_{k}\sum_{m}\left|\int P_K(\wt{u}_{N_1}\wt{u}_{N_2})(\Box^K_{-k}\wt{u}_{N_3})(\Box_{m}\Box^K_{k}\wt{v}_N)dxdt \right|\\
\les& {\sum_{K}}\sum_{k}\sum_{m} \norm{P_K(\wt{u}_{N_1}\wt{u}_{N_2})}_{\widehat{L}^{r}_{t,x}}\norm{\Box^K_{-k}\wt{u}_{N_3}}_{\widehat{L}^{r^{\prime}}_{t,x}}\norm{\Box_m\Box^K_{k} \wt{v}_N}_{\widehat{L}^{\infty}_{t,x}}\\
\les &{\sum_{K}}\sum_k N_1^{-1/r}K^{1/r^{\prime}-1/r}K^{1/r-1/r^{\prime}}(N_1^{2}K)^{-1/r-}K^{1/r^{\prime}-1/q}K^{1/q}\\
&\cdot\|\wt{u}_{N_1}\|_{\FX_{r,2r'}^{0,1/r+}} \|\wt{u}_{N_2}\|_{\FX_{r,2r'}^{0,1/r+}}\|\Box^K_{-k}\wt{u}_{N_3}\|_{\FX_{r, q}^{0,1/r+}}\|\Box^K_{k}\wt{v}_{N}\|_{\FX_{r^{\prime},q^{\prime}}^{0,1/r^{\prime}+}}\\
\les &\sum_{K}\sum_k N_1^{-3/r-} K^{1/r^{\prime}-1/r-}\max\{1, N_1^{1/r'-2/q} \}\\
&\cdot\|\wt{u}_{N_1}\|_{\FX_{r,q}^{0,1/r+}} \|\wt{u}_{N_2}\|_{\FX_{r,q}^{0,1/r+}}\|\Box^K_{-k}\wt{u}_{N_3}\|_{\FX_{r, q}^{0,1/r+}}\|\Box^K_{k}\wt{v}_{N}\|_{\FX_{r^{\prime},q^{\prime}}^{0,1/r^{\prime}+}}\\
\les& N_1^{-2s+2-4/r-2/q+\frac{1}{8r'}+} \|u_{N_1}\|_{\FX_{r,q}^{s,1/r+}} \|u_{N_2}\|_{\FX_{r,q}^{s,1/r+}}\|u_{N_3}\|_{\FX_{r, q}^{s,1/r+}}\|v_{N}\|_{\FX_{r^{\prime},q^{\prime}}^{-1-s,1/r^{\prime}-2\e}}.
\end{aligned}
\end{equation}
This bound is summable if {$s>1-\frac{2}{r}-\frac{1}{q}+\frac{1}{16r'}$}. Note that when $r' \leq q \leq 2r'$, we apply the embedding $\ell^q \subset \ell^{2r'}$ in the forth step; while for $2r' \leq q \leq \infty$, we apply the H\"older inequality.
 
{\bf Case 4: $N_1\sim N_2\sim N_3$.} 
	
If $N_1\sim N_2\sim N_3\gg N$, then it is similar to Case 3. 
However, a delicate analysis is needed for $N_1\sim N_2\sim N_3\sim N$.
Without loss of generality, we may assume by symmetry that the domain of the integration is given by 
\EQ{\Gamma^{{\prime}}=\Gamma\cap \{\sigma_{max}\ges N_1|(\xi_1+\xi)(\xi_3+\xi)|\}\cap \{|\xi_1+\xi|\leq |\xi_3+\xi|\}.}
Therefore, the integration turns to
\begin{equation}
\begin{aligned}		&\left|\int_{\Gamma^{\prime}}\wh{u}_{N_1}\wh{u}_{N_2}\wh{u}_{N_3}\wh{v}_{N}\right|\\
\leq&\bigg|\left(\int_{\Gamma^{\prime}, |\xi_3+\xi|\les 1}+\int_{\Gamma^{\prime}, |\xi_1+\xi|\les 1, |\xi_3+\xi|\ges 1}+\int_{\Gamma^{\prime}, |\xi_1+\xi|\ges 1, |\xi_3+\xi|\ges 1}\right)
\wh{u}_{N_1}\wh{u}_{N_2}\wh{u}_{N_3}\wh{v}_{N}\bigg|\\
=:&I+II+III.
\end{aligned}
\end{equation}
	
{$\bullet$} Term $I$. For this term, we perform box decomposition for each frequency and get from the trilinear estimate (b) in Lemma \ref{lem:trilinear-FL} and \eqref{stri-pq} that for $1 < r \leq 2$
\begin{equation}
\begin{aligned}
I\les& \sum_k \left|\int (\Box_k \wt{u}_{N_1})(\Box_{-k} \wt{u}_{N_2})(\Box_k \wt{u}_{N_3}) (\Box_{-k}\wt{v}_N)dxdt\right|\\
\les& \sum_k\left\|(\Box_k \wt{u}_{N_1})(\Box_{-k} \wt{u}_{N_2})(\Box_k \wt{u}_{N_3})\right\|_{\fL_{t,x}^{r}}\norm{\Box_{-k}\wt{v}_N}_{\fL^{r^{\prime}}_{t,x}}\\
\les&\left(\prod_{j=1}^3 N_j^{-1 / 3 r}\left\|\widetilde{u}_{N_j}\right\|_{\widehat{X}_{r, q}^{0, 1/r+}}\right)\left\|\widetilde{v}_N\right\|_{\widehat{X}_{r^{\prime},q^{\prime}}^{0, 1/r^{\prime}-2\e}}\\
\les& {N_1^{-1/r-2s+1} }\prod_{j=1}^3\norm{u_{N_j}}_{\FX_{r,q}^{s,1/r+}} \norm{v_N}_{\FX_{r^{\prime},q^{\prime}}^{-1-s,1/r^{\prime}-2\e}}.
\end{aligned}
\end{equation}
This bound is summable if {$s\geq \frac{1}{2}-\frac{1}{2r}$}. This is in fact the strongest condition. 
	
{$\bullet$} Term $II$. Since $|\xi_1+\xi|\les 1$ and $|\xi_3+\xi|\ges 1$, we may assume $|\xi_3+\xi|=|\xi_1+\xi_2| \sim K \ges 1$. Then the term $II$ yields
\EQ{\label{tri-II}
II\les& \sum_k \sum_{K\ges 1}\left|\int P_K(\Box_k \wt{u}_{N_1} \wt{u}_{N_2}) P_K (\Box_{-k}\wt{v}_N \wt{u}_{N_3})dxdt\right|.}
Since $|\xi_1+\xi_2|\sim K$ and $|\xi_1-k|\leq 1$, $\xi_2$ is located in an interval centred at $-k$ with length $\sim K$. 
Similarly, $|\xi_3+\xi|\sim K$ and $|\xi+k|\leq 1$, which implies that $\xi_3$ is located in an interval centred at $k$ with length $\sim K$. Therefore, it is equivalent to write $\wt{u}_{N_2}$ as $ \Box_{-k/K}^K\wt{u}_{N_2}$, and $ \wt{u}_{N_3}$ as $\Box_{k/K}^K\wt{u}_{N_3}$. By using \eqref{cor:mo-bi} and \eqref{Holder}, \eqref{tri-II} turns to
\begin{equation*}
\begin{aligned}
II\les& \sum_k {\sum_{K}}\left|\int P_K(\Box_k \wt{u}_{N_1} \Box_{-k/K}^K\wt{u}_{N_2}) P_K(\Box_{-k}\wt{v}_N  \Box_{k/K}^K \wt{u}_{N_3})dxdt\right|\\
\les& \sum_k {\sum_{K}}\norm{P_K(\Box_k \wt{u}_{N_1} \Box_{-k/K}^K\wt{u}_{N_2})}_{\widehat{L}^{r}_{t,x}}\norm{P_K(\Box_{-k}\wt{v}_N  \Box_{k/K}^K \wt{u}_{N_3})}_{\widehat{L}^{r^{\prime}}_{t,x}}\\
\les& \sum_k{\sum_{K}} N_1^{-1/r}K^{-1/r}N_1^{-2s+1}\norm{\Box_k \wt{u}_{N_1}}_{\FX_r^{s,1/r+}}\norm{\Box_{-k/K}^K \wt{u}_{N_2}}_{\FX_r^{s,1/r+}} \\
& \cdot N_1^{-1/r^{\prime}}K^{-1/r^{\prime}}\norm{\Box_{k/K}^K \wt{u}_{N_3}}_{\FX_{r^{\prime}}^{s,1/r^{\prime}+}}\norm{\Box_{-k}\wt{v}_N}_{\FX_{r^{\prime}}^{-1-s, 1/r^{\prime}+}}\\
\les& \sum_k {\sum_{K}} N_1^{-1+\frac{1}{8r'}+}K^{-1/r-1/r^{\prime}}N_1^{-2s+1}K^{1/r-1/r^{\prime}}K^{1/r^{\prime}-1/q}K^{1/r^{\prime}-1/q}\\
& \cdot\norm{ \Box_k \wt{u}_{N_1}}_{\FX_{r,q}^{s, 1/r+}}\norm{\Box_{-k/K}^K \wt{u}_{N_2}}_{\FX_{r,q}^{s, 1/r+}} \norm{\Box_{k/K}^K \wt{u}_{N_3}}_{\FX_{r,q}^{s, 1/r+}}\norm{ \Box_{-k} \wt{v}_N}_{\FX_{r^{\prime}, q^{\prime}}^{-1-s,1/r^{\prime}-}}\\
\les& {N_1^{\frac{1}{8r'}-2s+} }\norm{ u_{N_1}}_{\FX_{r,q}^{s,1/r+}}\norm{u_{N_2}}_{\FX_{r,q}^{s,1/r+}} \norm{u_{N_3}}_{\FX_{r,q}^{s,1/r+}}\norm{v_N}_{\FX_{r^{\prime}, q^{\prime}}^{-1-s,1/r^{\prime}-2\e}}.
\end{aligned}
\end{equation*}
This bound is summable if {$s>\frac{1}{16r'}$}.
	
{$\bullet$} Term $III$. We localize $|\xi_1+\xi|=|\xi_2+\xi_3|\sim K_1$ and $|\xi_3+\xi|=|\xi_1+\xi_2|\sim K_2$ for dyadic numbers $1 \lesssim K_1\leq K_2$, which implies $\sigma_{max}\ges N_1K_1K_2$.  
We may discuss the subcases $|\sigma_1|=\sigma_{max}$ and $|\sigma_2|=\sigma_{max}$, since the others are similar and easier.
	
{\underline{$|\sigma_1|=\sigma_{max}$.}} We have
\EQ{\label{III-1-explain}
III\les& \sum_k \sum_{1\lesssim K_1}\left|\int \Box_k^{K_1} \wt{u}_{N_1}P_{K_1}( \wt{u}_{N_2} \wt{u}_{N_3}) \Box_{-k}^{K_1}\wt{v}_N dxdt\right|.}
We observe that $\xi_1$ is located in an interval centred at $kK_1$ with length $\sim K_1$. Therefore, $|\xi_1+\xi_2|\sim K_2$ implies that $\xi_2$ is located in an interval centred at $-kK_1$ with length $\sim K_2$. 
Similarly, since $|\xi_3+\xi|\sim K_2$, then $\xi_3$ is located in an interval centred at $kK_1$ with length $\sim K_2$. 
Thus, it is equivalent to write $\wt{u}_{N_2}$ as $\Box_{-kK_1/K_2}^{K_2}\wt{u}_{N_2}$, and $\wt{u}_{N_3}$ as $\Box_{kK_1/K_2}^{K_2}\wt{u}_{N_3}$ in \eqref{III-1-explain}. Hence, we have
\begin{equation*}\label{III-1}
\begin{aligned}
\eqref{III-1-explain}\les& \sum_k {\sum_{K_1\leq K_2}}\left|\int (\Box_k^{K_1} \wt{u}_{N_1})P_{K_{1}}(\Box_{-kK_1/K_2}^{K_2}\wt{u}_{N_2} \Box_{kK_1/K_2}^{K_2} \wt{u}_{N_3}) (\Box_{-k}^{K_1}\wt{v}_N) dxdt\right|\\
\les& \sum_k {\sum_{K_1\leq K_2}}\norm{\Box_k^{K_1} \wt{u}_{N_1}}_{\widehat{L}^{r^{\prime}}_{t,x}}\norm{P_{K_1}(\Box_{-kK_1/K_2}^{K_2}\wt{u}_{N_2} \Box_{kK_1/K_2}^{K_2} \wt{u}_{N_3})}_{\widehat{L}^{r}_{t,x}}\\
&\cdot\norm{\sum_m \Box_m\Box_k^{K_1}\wt{v}_N}_{\widehat{L}^{\infty}_{t,x}}\\
\les& \sum_k {\sum_{K_1\leq K_2}} K_1^{1/r-1/r^{\prime}}(N_1K_1 K_2)^{-1/r-}K_1^{1/r^{\prime}-1/q}\norm{\Box_k^{K_1} \wt{u}_{N_1}}_{\FX_{r, q}^{0,1/r+}}\\
&\cdot {N_1^{-1/r}K_1^{1/r^{\prime}-1/r}\max\{1, K_2^{1/r^{\prime}-2/q}\}}\norm{\Box_{-kK_1/K_2}^{K_2}\wt{u}_{N_2}}_{\FX_{r,q}^{0,1/r+}}
\norm{\Box_{kK_1/K_2}^{K_2}{u}_{N_3}}_{\FX_{r,q}^{0,1/r+}}\\
&\cdot K_1^{1/q}\norm{\Box_k^{K_1}\wt{v}_N}_{\FX_{r^{\prime},q^{\prime}}^{0,1/r^{\prime}+}} \\
\les& \sum_k {\sum_{K_1\leq K_2}}{ N_1^{-2s+1-2/r+\frac{1}{8r'}+}}K_{1}^{1/r^{\prime}-1/r-}K_2^{1/r^{\prime}-1/r-2/q-}\norm{\Box_k^{K_1} \wt{u}_{N_1}}_{\FX_{r, q}^{s,1/r+}}\\
&\cdot\norm{\wt{u}_{N_2}}_{\FX_{r,q}^{s,1/r+}} \norm{\wt{u}_{N_3}}_{\FX_{r,q}^{s,1/r+}}\norm{\Box_k^{K_1}\wt{v}_N}_{\FX_{r^{\prime},q^{\prime}}^{-1-s,1/r^{\prime}-2\e}}\\
\les& {N_1^{-2s+1-2/r+\frac{1}{8r'}+}}\norm{ u_{N_1}}_{\FX_{r,q}^{s,1/r+}}\norm{u_{N_2}}_{\FX_{r,q}^{s,1/r+}} \norm{u_{N_3}}_{\FX_{r,q}^{s,1/r+}}\norm{v_N}_{\FX_{r^{\prime},q^{\prime}}^{-1-s,1/r^{\prime}-2\e}},
\end{aligned}
\end{equation*}
where we mainly used the improved bilinear estimate \eqref{improve-bi} in the third step. This bound is summable if $s>\frac{1}{2}-\frac{1}{r}+\frac{1}{16r'}$.
	
{\underline{$|\sigma_2|=\sigma_{max}$.}}   
If $K_{1} \leq K_{2}$, we need to localize $|\xi_1+\xi_3|\sim K_3$ if $\xi_1$ and $\xi_3$ have different signs. 
 
If $K_3\les K_1$, then $\xi_2$ and $\xi_3$ can be localized to intervals of size $K_1$. Indeed, since $\xi_1$ is located in an interval centred at $k K_1$ with length $\sim K_1$ as well as $|\xi_1+\xi_3|\sim K_3$, then $\xi_3$ is located to an interval centred at $-k K_1$ with length $\sim K_1$. Since $|\xi_2+\xi_3|\sim K_1$, then $\xi_2$ is located in an interval centred at $k K_1$ with length $ \sim K_1$. Hence, we have
\begin{equation}\label{III-2-explain}
\begin{aligned}
III
\les& \sum_k \sum_{K_1\leq K_2}\left|\int (\Box_k^{K_1} \wt{u}_{N_1}) (\Box_{k}^{K_1}\wt{u}_{N_2}) (\Box_{-k}^{K_1} \wt{u}_{N_3}) (\Box_{-k}^{K_1}\wt{v}_N) dxdt\right|.
\end{aligned}
\end{equation}
Applying the trilinear estimate (a) in Lemma \ref{lem:trilinear-FL} with $r=1$, we get
\begin{equation}\label{III_tri_1}
\begin{aligned}
\eqref{III-2-explain}\les& \sum_k {\sum_{K_1\leq K_2}}\norm{(\Box_k^{K_1} \wt{u}_{N_1}) (\Box_{k}^{K_1}\wt{u}_{N_2}) (\Box_{-k}^{K_1} \wt{u}_{N_3})}_{\widehat{L}^{1}_{t,x}}\norm{\sum_{m}\Box_{m}\Box_{-k}^{K_1}\wt{v}_N}_{\widehat{L}^{\infty}_{t,x}}\\
\les&\sum_k {\sum_{K_1\leq K_2}} N_{1}^{-1}\norm{\Box_k^{K_1} \wt{u}_{N_1}}_{\FX_{1, \infty}^{0+,1+}}\norm{\Box_{k}^{K_1}\wt{u}_{N_2}}_{\FX_{1, \infty}^{0,1+}} \norm{\Box_{-k}^{K_1}\wt{u}_{N_3}}_{\FX_{1, \infty}^{0,1+}} \norm{\Box^{K_1}_{-k}\wt{v}_N}_{\FX_{\infty, 1}^{0,0}}\\
\les&{N_{1}^{-2s+}\norm{{u}_{N_1}}_{\FX_{1, \infty}^{s,1+}}\norm{{u}_{N_2}}_{\FX_{1, \infty}^{s,1+}} \norm{{u}_{N_3}}_{\FX_{1, \infty}^{s,1+}}
\norm{{v}_N}_{\FX_{\infty,1}^{-1-s,0}}}.
\end{aligned}
\end{equation}
On the other hand, using the improved $L^{4}$-estimate \eqref{L4-2}, we obtain for any $q\geq 2$
\begin{equation}\label{III_tri_2}
\begin{aligned}
\eqref{III-2-explain}\les& \sum_k {\sum_{K_1\leq K_2}}\norm{(\Box_k^{K_1} \wt{u}_{N_1}) (\Box_{k}^{K_1}\wt{u}_{N_2}) (\Box_{-k}^{K_1} \wt{u}_{N_3})}_{{L^{1}_{t,x}}} \norm{\sum_{m}\Box_{m}\Box_{-k}^{K_1}\wt{v}_N}_{{L^{\infty}_{t,x}}}\\
\les& \sum_k {\sum_{K_1\leq K_2}}\norm{\Box_k^{K_1} \wt{u}_{N_1}}_{{L^{4}_{t,x}}} \norm{\Box_{k}^{K_1}\wt{u}_{N_2}}_{{L^{2}_{t,x}}} \norm{\Box_{-k}^{K_1} \wt{u}_{N_3}}_{{L^{4}_{t,x}}} \norm{\sum_{m}\Box_{m}\Box_{-k}^{K_1}\wt{v}_N}_{{L^{\infty}_{t,x}}}\\
\les& \sum_k {\sum_{K_1\leq K_2}} N_1^{-1/8+}\norm{\Box_k^{K_1} \wt{u}_{N_1}}_{\FX_{2,4}^{0, 1/2+}}(N_1K_1K_2)^{-1/2-}\norm{\Box_{k}^{K_1}\wt{u}_{N_2}}_{\FX_{2}^{0,1/2+}}\\
&\cdot N_1^{-1/8+} \norm{\Box_{-k}^{K_1}\wt{u}_{N_3}}_{\FX_{2, 4}^{0, 1/2+}}K_1^{1/q}\norm{\Box^{K_1}_{-k}\wt{v}_N}_{\FX_{2,q^{\prime}}^{0,1/2+}}\\
\les &\sum_k {\sum_{K_1\leq K_2}} N_1^{-2s+1/4+} K_1^{1/q}(K_1K_2)^{-1/2-}\norm{\Box_k^{K_1} \wt{u}_{N_1}}_{\FX_{2,4}^{s, 1/2+}}\norm{\Box_{k}^{K_1}\wt{u}_{N_2}}_{\FX_{2}^{s,1/2+}}\\
&\cdot \norm{\Box_{-k}^{K_1}\wt{u}_{N_3}}_{\FX_{2, 4}^{s, 1/2+}}\norm{\Box^{K_1}_{-k}\wt{v}_N}_{\FX_{2,q^{\prime}}^{-1-s,1/2+}}\\
\les &\sum_k {\sum_{K_1\leq K_2}} N_1^{-2s+1/4+} K_1^{1/q}(K_1K_2)^{-1/2-}\max\{1,K_1^{1/2-2/q}\}K_1^{1/2-1/q}\\
&\cdot \norm{\Box_k^{K_1} \wt{u}_{N_1}}_{\FX_{2,q}^{s, 1/2+}}\norm{\Box_{k}^{K_1}\wt{u}_{N_2}}_{\FX_{2,q}^{s,1/2+}}\norm{\Box_{-k}^{K_1}\wt{u}_{N_3}}_{\FX_{2, q}^{s, 1/2+}}\norm{\Box^{K_1}_{-k}\wt{v}_N}_{\FX_{2,q^{\prime}}^{-1-s,1/2+}}\\
\les& N_1^{-2s+1/4+} \norm{{u}_{N_1}}_{\FX_{2,q}^{s, 1/2+}}\norm{{u}_{N_2}}_{\FX_{2,q}^{s,1/2+}}\norm{{u}_{N_3}}_{\FX_{2, q}^{s, 1/2+}}\norm{{v}_N}_{\FX_{2,q^{\prime}}^{-1-s,1/2+}}.
\end{aligned}
\end{equation}
By the interpolation between \eqref{III_tri_1} and \eqref{III_tri_2}, we obtain for $1\leq r \leq 2$, $q\geq r'$
\EQN{\label{III-2}
\eqref{III-2-explain}
\les& {N_{1}^{-2s+\frac{5}{8r'}+}}\norm{{u}_{N_1}}_{\FX_{r, q}^{s, 1/r+}} \norm{{u}_{N_2}}_{\FX_{r, q}^{s, 1/r+}}\norm{{u}_{N_3}}_{\FX_{r, q}^{s,1/r+}}
\norm{{v}_N}_{\FX_{r^{\prime}, q^{\prime}}^{-1-s,1/r^{\prime}-2\e}}.
}
Hence, this bound is summable if $s>\frac{5}{16 r'}$.

If $K_3\ges K_1$, then we can localize $\xi_2$ and $\xi_3$ to intervals of size $\min(K_2,K_3)$. Without loss of generality, we assume $K_3\leq K_2$. Indeed, since $|\xi_1+\xi_3|\sim K_3$, and $\xi_1$ is located in an interval centred at $k K_1$ with length $\sim K_1$, then $\xi_3$ is located to an interval centred at $-k K_1$ with length $\sim K_3$. On the other hand, since $|\xi_2+\xi_3|\sim K_1$, then $\xi_2$ is located to interval centred at $k K_1$ with length $\sim K_3$. Therefore, we can write $\wt{u}_{N_2}$ as $\Box_{k K_{1}/K_3}^{K_3}\wt{u}_{N_2}$ and $\wt{u}_{N_3}$ as  $\Box_{-k K_1/K_3}^{K_3}\wt{u}_{N_3}$. Then, by Lemma \ref{lem:bilinear-FL} we see that
\EQN{\label{III-3}
III\les& \sum_k \sum_{K_3 \gtrsim K_1} \sum_{ K_1 \leq K_2} \left|\int (\Box_k^{K_1} \wt{u}_{N_1})( \Box_{k K_1/K_3}^{K_3}\wt{u}_{N_2})( \Box_{-k K_1/K_3}^{K_3} \wt{u}_{N_3})( \Box_{-k}^{K_1}\wt{v}_N) dx dt\right|\\
\les& \sum_k \sum_{K_3 \gtrsim K_1} \sum_{ K_1 \leq K_2} \norm{{P_{K_3}}(\Box_k^{K_1} \wt{u}_{N_1}\Box_{-k K_1/K_3}^{K_3} \wt{u}_{N_3})}_{\widehat{L}^{r}_{t,x}} \norm{\Box_{k K_1/K_3}^{K_3}\wt{u}_{N_2}}_{\widehat{L}^{r^{\prime}}_{t,x}}\\
&\cdot\norm{\sum_m\Box_m\Box_{-k}^{K_1}\wt{v}_N}_{\widehat{L}^{\infty}_{t,x}}\\
\les& \sum_k\sum_{K_3 \gtrsim K_1} \sum_{ K_1 \leq K_2}N_1^{-1/r}K_3^{-1/r}K_1^{1/r^{\prime}-1/q}K_3^{1/r^{\prime}-1/q}\norm{\Box_k^{K_1}{u}_{N_1}}_{\FX_{r,q}^{0,1/r+}}\\
&\cdot \norm{\Box_{-k K_1/K_3}^{K_3}{u}_{N_3}}_{\FX_{r,q}^{0,1/r+}}K_3^{1/r-1/r^{\prime}}(N_1K_1K_2)^{-1/r-}K_3^{1/r^{\prime}-1/q}\norm{\Box_{k K_1/K_3}^{K_3}{u}_{N_2}}_{\FX_{r, q}^{0,1/r+}}\\
&\cdot K_1^{1/q} \norm{\Box_{-k}^{K_1}{v}_N}_{\FX^{0,1/r^{\prime}+}_{r^{\prime},q^{\prime}}}\\
\les& {N_1^{-2s+1-2/r+\frac{1}{8r'}}}\norm{u_{N_1}}_{\FX_{r,q}^{s,1/r+}}\norm{u_{N_2}}_{\FX_{r, q}^{s,1/r+}} \norm{u_{N_3}}_{\FX_{r,q}^{s,1/r+}}\norm{v_N}_{\FX_{r^{\prime},q^{\prime}}^{-1-s,1/r^{\prime}-2\e}}.}
This bound is summable if {$s>\frac{1}{2}-\frac{1}{r}+\frac{1}{16r'}$}. 

Therefore, by taking the intersection of all the ranges of $s$, we see the condition is $s\geq \frac{1}{2r'}$. This completes the proof of Proposition \ref{prop:tri}.

\section{Ill-posedness}
In this section, we prove the ill-posedness result, which is equivalent to proving the following proposition:
\begin{proposition}
Under the assumptions of Theorem \ref{th1}, let $s < \frac{1}{2}-\frac{1}{2r}$
and $0<\delta \ll 1$. Then the solution map $\delta u_0 \rightarrow u_\delta(x, t)$ of the equation \eqref{eq:mkdv} in $\gFL^{s}_{r, q}$ is not $C^3$ continuous at the origin.
\end{proposition}
\begin{proof}
We define the solution map for the defocusing case (the focusing case can also be treated by the same line) as follows:
\begin{equation*}
\mathcal{T}: \quad \delta u_0 \rightarrow u_\delta(x, t)=W(t) \delta u_0+\int_0^t W(t-s)(u^3)_x(s) d s.
\end{equation*}
By straightforward computations, we have
\begin{equation}\label{iteration}
\begin{aligned}
& \left.\frac{\partial u}{\partial \delta}\right|_{\delta=0}=W(t) u_0=:u_1, \\
& \left.\frac{\partial^2 u}{\partial \delta^2}\right|_{\delta=0}=0=:u_2,\\
& \left.\frac{\partial^3 u}{\partial \delta^3}\right|_{\delta=0}=6 \int_0^t W(t-s) \partial_x(W(s) u_0)^3 d s=:u_3.
\end{aligned}
\end{equation}
It is well known that if the map $\delta u_0 \rightarrow u_\delta$ is of class $C^3$ at the origin, the necessary condition is
\begin{equation}\label{data}
\sup _{t \in[0, T]}\left\|u_3\right\|_{\gFL^{s}_{r,q}} \leq C\left\|u_0\right\|_{\gFL^{s}_{r,q}}^3.
\end{equation}
We seek for data $u_{0}$ such that \eqref{data} fails. 

Since the condition $s\geq \frac{1}{2r'}$ is from the resonant case, we construct an example of that case.
Consider $u_0 \in \gFL^{s}_{r,q}, s<1/2-1/2r$ defined by
\begin{equation}
\widehat{u_0}(\xi)=N^{-s}\beta^{-{1}/{r^\prime}}(\chi_{[-N-\beta, -N]}(\xi)+\chi_{[N, N+\beta]}(\xi)),\ N \gg 1,
\end{equation}
where $\beta=\beta(N)=2^{-10}N^{-1/2}<1$.
Note that $\left\|u_0\right\|_{\gFL^{s}_{r,q}} \sim 1.$
We estimate the Fourier transform of $u_3$ in \eqref{iteration} as follows
\EQN{
&\ft_{x}(u_3)(\xi,t)\\
= &C \int_0^t e^{i(t-s) \omega(\xi)}(i\xi)\int_{\mathbb{R}^2} e^{is\omega\left(\xi-\xi_1-\xi_2\right)} \widehat{u_0}(\xi-\xi_1-\xi_2) e^{i s \omega(\xi_1)} \widehat{u_0}(\xi_1)e^{i s \omega(\xi_2)} \widehat{u_0}(\xi_2) d \xi_1d \xi_2 ds\\
= & e^{i t \omega\left(\xi\right)}(i\xi)\int_{\mathbb{R}^2} \int_0^t e^{is \Psi(\xi, \xi_1,\xi_2)} d s \widehat{u_0}(\xi-\xi_1-\xi_2) \widehat{u_0}(\xi_1)\widehat{u_0}(\xi_2) d \xi_1 d \xi_2 \\
= & e^{i t \omega(\xi)}(i\xi)\int_{\mathbb{R}^2} \frac{e^{i t \Psi(\xi, \xi_1, \xi_2)}-1}{i\Psi(\xi, \xi_1, \xi_2)} \widehat{u_0}(\xi-\xi_1-\xi_2) \widehat{u_0}(\xi_1)\widehat{u_0}(\xi_2) d \xi_1 d \xi_2,
}
where $\Psi\left(\xi, \xi_1, \xi_2\right)=-3(\xi-\xi_1)(\xi-\xi_2)(\xi_1+\xi_2)$. 

We assume $\xi \in [N+\beta/3, N+\beta/2]$. Then we have either $\xi_1, -\xi_2\in [-N-\beta, -N]$ or $-\xi_1, \xi_2\in [-N-\beta, -N]$, otherwise $u_3=0$. For $\xi_1,-\xi_2\in [-N-\beta, -N]$ or $-\xi_1,\xi_2\in [-N-\beta, -N]$, we have
\EQN{
|\Psi(\xi,\xi_1,\xi_2)|\ll N \beta^2\sim 1.
}
Therefore, for $\xi \in [N+\beta/3, N+\beta/2]$
\EQN{
|\widehat{u_3}(t, \xi)|
\sim N^{-3s}\beta^{-{3}/{r^\prime}}N t \beta^2
}
and so
\EQN{
\left\|u_3\right\|_{\gFL_{r, q}^{s}} \geq C tN^{-2 s+1}\beta^{-{3}/{r^\prime}+2+{1}/{r^\prime}} \gtrsim t N^{-2 s+1-1/r}.
}
Taking $N\gg 1$, we see \eqref{data} does not hold if 
$s<\frac{1}{2r'}$.
\end{proof}

\end{document}